\documentclass[a4paper,12pt]{article}
\usepackage{tikz}
\usetikzlibrary{matrix,arrows,decorations.markings}
\usepackage{amsmath,amsfonts,amssymb,amsthm,url,textcomp}
\usepackage{csquotes}
\usepackage{a4wide}
\usepackage[numbers]{natbib}
\usepackage{fancyhdr}
\usepackage{anyfontsize}
\usepackage{hyperref}
\usepackage{hhline}
\usepackage{leftidx}

\allowdisplaybreaks

\makeatletter
\let\@fnsymbol\@arabic
\makeatother

\newtheorem{theorem}{Theorem}\numberwithin{theorem}{section}
\newtheorem{definition}[theorem]{Definition}

\newtheorem{corollary}[theorem]{Corollary}
\newtheorem{proposition}[theorem]{Proposition}

\newtheorem{conjecture}[theorem]{Conjecture}

\newtheorem{theoremm}{Theorem}\numberwithin{theoremm}{subsection}
\newtheorem{deffinition}[theoremm]{Definition}
\newtheorem{lemmma}[theoremm]{Lemma}

\newtheorem{nottation}[theoremm]{Notation}
\newtheorem{propposition}[theoremm]{Proposition}

\numberwithin{theoremmm}{subsubsection}

\theoremstyle{remark}

\newtheorem{example}[theorem]{Example}
\newtheorem{remmark}[theoremm]{Remark}

\newcommand{\Rad}{\operatorname{Rad}}
\newcommand{\Aut}{\operatorname{Aut}}

\newcommand{\PSL}{\operatorname{PSL}}

\newcommand{\lcm}{\operatorname{lcm}}

\newcommand{\ord}{\operatorname{ord}}
\newcommand{\Sym}{\operatorname{Sym}}

\newcommand{\Soc}{\operatorname{Soc}}

\newcommand{\id}{\operatorname{id}}

\newcommand{\fix}{\operatorname{fix}}
\newcommand{\Out}{\operatorname{Out}}
\newcommand{\e}{\mathrm{e}}

\newcommand{\PGL}{\operatorname{PGL}}

\newcommand{\GL}{\operatorname{GL}}
\newcommand{\Frob}{\operatorname{Frob}}

\newcommand{\G}{\mathcal{G}}

\newcommand{\D}{\operatorname{D}}

\newcommand{\Mod}[1]{\ (\textup{mod}\ #1)}

\newcommand{\F}{\operatorname{F}}

\newcommand{\E}{\operatorname{E}}
\newcommand{\Var}{\operatorname{Var}}

\newcommand{\IN}{\mathbb{N}}

\newcommand{\IF}{\mathbb{F}}

\newcommand{\Inndiag}{\operatorname{Inndiag}}

\newcommand{\Outdiag}{\operatorname{Outdiag}}

\newcommand{\supp}{\operatorname{supp}}

\newcommand{\SB}{\operatorname{SB}}

\newcommand{\Perm}{\operatorname{Perm}}
\newcommand{\MPO}{\operatorname{MPO}}

\newcommand{\BAD}{\operatorname{BAD}}
\newcommand{\Suz}{\operatorname{Suz}}
\newcommand{\Supp}{\operatorname{Supp}}

\begin{document}

\title{Words, permutations, and the nonsolvable length of a finite group}

\author{Alexander Bors\thanks{The University of Western Australia, Centre for the Mathematics of Symmetry and Computation, 35 Stirling Highway, Crawley 6009, WA, Australia. \newline E-mail: \href{mailto:alexander.bors@uwa.edu.au}{alexander.bors@uwa.edu.au} \newline The first author is supported by the Austrian Science Fund (FWF), project J4072-N32 \enquote{Affine maps on finite groups}.} \and Aner
 Shalev\thanks{Einstein Institute of Mathematics, Hebrew University, Jerusalem 91904, Israel. \newline E-mail: \href{mailto:shalev@math.huji.ac.il}{shalev@math.huji.ac.il} \newline The second author acknowledges the support of ISF grant 686/17, BSF grant 2016072 and the Vinik chair of mathematics which he
holds. \newline 2010 \emph{Mathematics Subject Classification}: Primary: 20E10, 20P05. Secondary: 20B05, 20D06, 20F22. \newline \emph{Key words and phrases:} Finite groups, Nonsolvable length, Probabilistic identities, Identities, Word maps}}

\date{\today}

\maketitle

\abstract{We study the impact of certain identities and probabilistic identities on the structure of finite groups.
More specifically, let $w$ be a nontrivial word in $d$ distinct variables and let $G$ be a finite group for which the word map $w_G:G^d\rightarrow G$ has a fiber of size at least $\rho|G|^d$ for some fixed $\rho > 0$. We show that, for certain words $w$, this implies that
$G$ has a normal solvable subgroup of index bounded above in terms of $w$ and $\rho$. We also show that, for a larger family
of words $w$, this implies that the nonsolvable length of $G$ is bounded above in terms of $w$ and $\rho$, thus
providing evidence in favor of a conjecture of Larsen.
Along the way we obtain results of some independent interest, showing roughly that most elements of large finite permutation groups have large support.}

\newpage

\section{Introduction}\label{sec1}

The impact of identities on the structure of groups has been a central research topic for over a century.
Major examples include the Burnside problems and their solutions, the theory of group varieties, as well as
parts of combinatorial and geometric group theory.

In the realm of finite groups, Zelmanov's solution to the Restricted Burnside Problem bounds the order of a $d$-generator
finite group satisfying the power identity $x^n \equiv 1$ in terms of $d$ and $n$ \cite{Z90, Z91}.
The Hall-Higman reduction of this problem to $p$-groups involves bounding the $p$-length of solvable groups satisfying this identity for all primes $p$ \cite{HH56a}. A recent related result of Segal bounds the generalized
Fitting height of finite groups satisfying $x^n \equiv 1$ in terms of $n$ \cite[Theorem 10]{Seg18a}.

More generally, in recent years there has been extensive interest in probabilistic identities (defined below) of finite
and residually finite groups. Finitely generated linear groups which satisfy a probabilistic identity were shown in \cite{LS16}
to be virtually solvable. Arbitrary residually finite groups satisfying a probabilistic identity were shown in \cite{LS18a}
(using results from \cite{Bor17}) to have nonabelian upper composition factors of bounded size. Probabilistically nilpotent finite
and infinite groups were recently studied in \cite{Sh18} and in \cite{MTVV18}.

It is easy to see that every finite group $G$ has a normal series each of whose factors is either solvable or a direct product of nonabelian finite simple groups. The smallest number of nonsolvable factors in a shortest such series is defined by Khukhro and Shumyatsky in \cite{KS15a} to be the \emph{nonsolvable length of $G$}, and is denoted by $\lambda(G)$ (see also Section \ref{sec2} below for an alternative definition, which was also already given in \cite[first paragraph of the Introduction]{KS15a}); while this concept was explicitly introduced and studied in \cite{KS15a}, the idea of writing a finite group $G$ as an extension of two finite groups with smaller nonsolvable lengths for inductive purposes is already implicit in the Hall-Higman paper, see \cite[proof of Theorem 4.4.1]{HH56a}.

The main purpose of this paper is to present some ideas relating identities and probabilistic identities in finite groups with the nonsolvable length, and sometimes with the index of the solvable radical. We combine some machinery already developed by the first author in
\cite{Bor17a} (building on earlier work of Nikolov from \cite{Nik16a}) with some new methods. Let us now explain this in some more detail.

For a positive integer $d$, denote by $\F(X_1,\ldots,X_d)$ the free group freely generated by $X_1,\ldots,X_d$. Elements of these groups are called \emph{words}. For the definition of \emph{probabilistic identity}, let $w\in\F(X_1,\ldots,X_d)$ be a nontrivial word. Then for every (not necessarily finite) group $G$, one has the word map $w_G:G^d\rightarrow G$, induced by substitution into $w$. If $G$ is finite and $g\in G$, it makes sense to define
\[
p_{w,G}(g):=\frac{1}{|G|^d}|\{(g_1,\ldots,g_d)\in G^d\mid w_G(g_1,\ldots,g_d)=g\}|,
\]
the proportion in $G^d$ of the fiber of $g$ under $w_G$. For profinite groups $G$, $p_{w,G}(g)$ denotes the (normalized) Haar measure (in $G^d$)
of the fiber $w_G^{-1}(g)$. We say that \emph{$G$ satisfies a probabilistic identity with respect to $w$ and $\rho\in\left(0,1\right]$} if and only if there is an element $g\in G$ such that $p_{w,G}(g)\geq\rho$. A residually finite group is said to satisfy a probabilistic identity if its
profinite completion satisfies a probabilistic identity.

In this paper, we will be interested in the following property of nontrivial words:

\begin{definition}\label{nlbDef}
A nontrivial word $w\in\F(X_1,\ldots,X_d)$ is called \emph{nonsolvable-length-bounding} (or NLB for short) if and only if there is a function $f_w:\left(0,1\right]\rightarrow\left[0,\infty\right)$ such that for every $\rho\in\left(0,1\right]$ and every finite group $G$, if $G$ satisfies a probabilistic identity with respect to $w$ and $\rho$, then $\lambda(G)\leq f_w(\rho)$.
\end{definition}

We can now state the following.

\begin{conjecture}\label{conj1} All nontrivial words are NLB.
\end{conjecture}

This conjecture, due to Michael Larsen (private communication), seems very challenging, in view of the fact that it is even unknown
for $\rho = 1$, namely when $w$ is an identity of $G$.

\begin{conjecture}\label{conj2} The nonsolvable length of a finite group which satisfies a nontrivial identity
$w \equiv 1$ is bounded above in terms of $w$.
\end{conjecture}

See also the last paragraph of \cite{Seg18a} for a related problem, where the nonsolvable length is replaced by the
generalized Fitting height.

Conjecture \ref{conj2} is reduced to bounding the Fitting height $h(G)$ of finite solvable groups $G$ satisfying a nontrivial
identity $w$ in terms of $w$ alone; indeed this reduction follows from \cite[Corollary 1.2]{KS15a}.

In \cite{Bor17a}, the first author studied another property of nontrivial words $w$, that of being \emph{multiplicity-bounding} (or \emph{MB} for short), see \cite[Definition 1.1.1]{Bor17a}. This just means that if a finite group $G$ satisfies a probabilistic identity with respect to $w$ and $\rho\in\left(0,1\right]$, then for each nonabelian finite simple group $S$, the multiplicity of $S$ as a composition factor of $G$ is bounded from above in terms of $w$, $\rho$ and $S$. Several stronger and weaker properties than that of being MB were also studied in \cite{Bor17a}, such as the ones in the last two enumeration points of the following definition:

\begin{definition}\label{veryStronglyMBDef}
Let $w\in\F(X_1,\ldots,X_d)$ be a nontrivial word.
\begin{enumerate}
\item A \emph{variation of $w$} is a word obtained from $w$ by \enquote{splitting variables}, i.e., by adding, for each $i\in\{1,\ldots,d\}$, to each occurrence of $X_i^{\pm1}$ in $w$ some second index.
\item For a nonabelian finite simple group $S$, we say that \emph{$w$ is a coset identity over $S$} if and only if there are $\alpha_1,\ldots,\alpha_d\in\Aut(S)$ such that $w_{\Aut(S)}$ is constant on the Cartesian product $\prod_{i=1}^d{S\alpha_i}$ of cosets of $S$ in $\Aut(S)$.
\item $w$ is called \emph{weakly multiplicity-bounding} (or \emph{WMB} for short) if and only if $w$ is not a coset identity over
any nonabelian finite simple group.
\item $w$ is called \emph{very strongly multiplicity-bounding} (or \emph{VSMB} for short) if and only if every variation of $w$ is WMB.
\end{enumerate}
\end{definition}

Note that our definition of a variation slightly differs from the one in \cite[Definition 2.4(2)]{Bor17a}, which included a technical restriction on the second indices which one can assume w.l.o.g.~anyway, but we will not need it here. 

By \cite[Theorem 5.2]{LS18a}, if a finite group $G$ satisfies a probabilistic identity with respect to $w$ and $\rho$, then the orders of the nonabelian composition factors of $G$ are bounded from above in terms of $w$ and $\rho$. Letting $\Rad(G)$ denote the solvable radical of a finite group $G$ (namely the largest solvable normal subgroup of $G$), this implies the following.

\begin{corollary}\label{mbRadCor}
A nontrivial word $w$ is MB if and only if the assumption that a finite group $G$ satisfies a probabilistic identity with respect to $w$ and $\rho$ implies that the radical index $[G:\Rad(G)]$ is bounded from above in terms of $w$ and $\rho$. In particular if $w$ is MB then it is NLB.
\end{corollary}

The proof of this result will be given in Subsection \ref{subsec4P2} for the reader's convenience.

Hence \cite[Theorem 1.1.2]{Bor17a} provides us with some examples of NLB words. Also by \cite[Theorem 1.1.2(1)]{Bor17a}, the shortest nontrivial words which are not MB are of the form $x^8$ where $x$ is a variable. We will, however, be able to show that such words are NLB, and the crucial observation is that while these words are not MB, in particular not VSMB, they are \enquote{almost} VSMB, in the following exact sense:

\begin{definition}\label{almostDef}
Let $w\in\F(X_1,\ldots,X_d)$ be a nontrivial word. $w$ is called \emph{almost very strongly multiplicity-bounding} (or \emph{almost VSMB} for short) if and only if every \emph{proper} variation $w'$ of $w$ (i.e., such that the number of variables occurring in $w'$ is strictly larger than the number of variables in $w$) is WMB.
\end{definition}

Our first main result relates the concepts of almost VSMB and NLB words:

\begin{theorem}\label{mainTheo1}
Almost VSMB words are NLB.
\end{theorem}

Thus almost VSMB words satisfy Conjecture \ref{conj1}.
This theorem is proved using a result of independent interest, showing that \emph{if $P$ is a finite permutation
group, and the proportion of elements $\sigma \in P$ whose support has size at most $C$ is at least $\rho > 0$,
then $|P|$ is bounded above in terms of $C$ and $\rho$}. See Theorem \ref{permGroupTheo} below, as well as
Theorem \ref{supp} and Proposition \ref{betterbound} for related results on permutation groups and the 
support of their elements.

Using the above result, Corollary \ref{mbRadCor} and \cite[Theorem 1.1.2(3)]{Bor17a}, the following is immediate:

\begin{corollary}\label{mainCor1}
Let $w\in\F(X_1,\ldots,X_d)$ be a nontrivial word of length at most $8$. Then $w$ is NLB.
\end{corollary}

Theorem \ref{mainTheo1} and Corollary \ref{mainCor1} provide evidence in favor of Larsen's conjecture mentioned above. We note that while $X_1^{12}$ is also not MB, the authors cannot exclude the possibility that all words of lengths $9$, $10$ and $11$ are VSMB, in particular NLB, thus possibly allowing
to replace the constant $8$ in Corollary \ref{mainCor1} by $11$. However, compared to studying words of lengths up to $8$ as done by the first author in \cite[Section 6]{Bor17a}, the computational cost of doing so even just for words of length $9$ is considerable and would most likely require a medium- to large-scale parallel computation. Still, with some more theoretical machinery, we will at least be able to show the following:

\begin{corollary}\label{mainCor2}
Let $w\in\F(X_1,\ldots,X_d)$ be a nontrivial word of length at most $11$. Then there is a constant $L_w\in\IN$ such that if a finite group $G$ satisfies the identity $w\equiv1$, then $\lambda(G)\leq L_w$.
\end{corollary}

Thus words of length at most $11$ satisfy Conjecture \ref{conj2}.
The proof of Corollary \ref{mainCor2} is based on a result allowing one to infer, under certain assumptions on a nontrivial word $w$, that if a finite group $H$ without nontrivial solvable normal subgroups satisfies the identity $w\equiv1$, then the so-called \emph{permutation part of $H$} (see Definition \ref{permDef}(1) below) satisfies a shorter identity. This result is formulated in detail in Subsection \ref{subsec5P1} as Theorem \ref{mainTheo2}.

Apart from new techniques for relating (probabilistic) identities with the nonsolvable length, we will also give infinitely many new (i.e., not already implicit in \cite[Theorem 1.1.2(1)]{Bor17a}) examples of both MB and non MB words. Recall that the power words $x^e$ are MB for all odd $e$ (as shown in the above reference). Answering \cite[Question 7.1]{Bor17a} we show the following.

\begin{theorem}\label{newMBTheo}
Let $x$ be a variable. Then the following hold:
\begin{enumerate}
\item Let $m$ be a positive integer such that every prime divisor $l$ of $m$ satisfies $l\equiv1\Mod{225}$. Then $x^{2m}$ is MB.
\item Let $e$ be a positive integer with $e>4$ and $4\mid e$. Then $x^e$ is not MB.
\end{enumerate}
\end{theorem}

Obtaining a better understanding for which positive integers $e$ the word $x^e$ is (or is not) MB is of intrinsic interest, but it also relates to bounding $\lambda(G)$ in terms of the group exponent $\exp(G)$, see Subsection \ref{subsec6P1}. We note that Theorem \ref{newMBTheo}(2) partially contradicts the first author's result \cite[Theorem 1.1.2(1)]{Bor17a}; more precisely, \cite[Theorem 1.1.2(1)]{Bor17a} wrongly states that $x^{20}$ is MB, but it is not. However, as clarified in an erratum on \cite{Bor17a} prepared by the first author, \cite[Theorem 1.1.2(1)]{Bor17a} does become true if one replaces the set $\{8,12,16,18\}$ in its statement by $\{8,12,16,18,20\}$ (so $20$ is the only exponent $e$ for which the original version of \cite[Theorem 1.1.2(1)]{Bor17a} makes a wrong statement on the MB property status of $x^e$). Except for the paragraph at hand, whenever we cite \cite[Theorem 1.1.2(1)]{Bor17a} in our paper (as we already did above), we are actually always referring to the above mentioned corrected version of it.

This paper is organized as follows. In Section \ref{sec2} we introduce some notation.
Section \ref{sec3} is devoted to permutation groups and the supports of its elements.
We obtain there results of independent interest, some of which are applied in subsequent sections.
In Section \ref{sec4} we study probabilistic identities and prove Theorem \ref{mainTheo1} and Corollary
\ref{mainCor1}. Section \ref{sec5} is devoted to identities and the proof of Corollary \ref{mainCor2}. Finally, in Section \ref{sec6}, we prove Theorem \ref{newMBTheo} as well as a few other results on the impact of power word identities on the group structure.
In particular we show there that the nonsolvable length of a finite group is bounded above by the exponent of its
Sylow $2$-subgroups.

\section{Some notation and prerequisites}\label{sec2}

We first discuss an equivalent, but more explicit (though also more technical) definition of $\lambda(G)$.

\begin{definition}\label{lambdaDef}
Let $G$ be a finite group.
\begin{enumerate}
\item We denote by $\Rad(G)$ the \emph{solvable radical of $G$}, the largest solvable normal subgroup of $G$.
\item We denote by $\Soc(G)$ the \emph{socle of $G$}, the subgroup of $G$ generated by all the minimal normal subgroups of $G$.
\item We define sequences $(G_k(G))_{k\geq1}$, $(R_k(G))_{k\geq1}$, $(H_k(G))_{k\geq1}$ and $(T_k(G))_{k\geq1}$ of characteristic sections of $G$ recursively as follows:
\begin{enumerate}
\item $G_1(G):=G$.
\item For $k\geq1$, $R_k(G):=\Rad(G_k(G))$.
\item For $k\geq1$, $H_k(G):=G_k(G)/R_k(G)$.
\item For $k\geq1$, $T_k(G):=\Soc(H_k(G))$.
\item For $k\geq2$, $G_k(G):=H_{k-1}(G)/T_{k-1}(G)$.
\end{enumerate}
\end{enumerate}
\end{definition}

We call a finite group $H$ \emph{semisimple} if and only if $\Rad(H)$ is trivial, i.e., if and only if $H$ has no nontrivial solvable normal subgroups. For the basic structure theory of finite semisimple groups (from which several of the subsequently listed facts follow), see \cite[pp.~89ff.]{Rob96a}.

For every finite group $G$, the groups $R_k(G)$ are by definition all solvable, the groups $H_k(G)$ are semisimple, and the groups $T_k(G)$ are direct products of nonabelian finite simple groups. Moreover, since $H_k(G)$ embeds into the automorphism group of $T_k(G)$, we have that $T_k(G)$ is trivial if and only if $H_k(G)$ is trivial, so there is a unique non-negative integer $\lambda'(G)$ such that $T_1(G),\ldots,T_{\lambda'(G)}(G)$ are all nontrivial and $T_k(G)=\{1\}$ for $k>\lambda'(G)$. Actually, $\lambda'(G)=\lambda(G)$, by \cite[first paragraph in the Introduction]{KS15a}.

We now introduce some more notation and terminology that will be used throughout the paper. We denote by $\IN$ the set of natural numbers (including $0$) and by $\IN^+$ the set of positive integers. When $f:X\rightarrow Y$ is a function and $M\subseteq X$, then $f_{\mid M}$ denotes the restriction of $f$ to $M$, and $f[M]$ denotes the element-wise image of $M$ under $f$. Euler's constant will be denoted by $\e$ (which is to be distinguished from the variable $e$). For $c>1$, we denote by $\log_c$ the base $c$ logarithm, and $\log$ denotes $\log_{\e}$. For a set $\Omega$, the symmetric group on $\Omega$ is denoted by $\Sym(\Omega)$, and for $n\in\IN^+$, $\Sym(n)$ denotes the symmetric group on $\{1,\ldots,n\}$. The group of units of a field $K$ is denoted by $K^{\ast}$, and the algebraic closure of $K$ by $\overline{K}$. For a prime power $q$, the finite field with $q$ elements is denoted by $\IF_q$. For a subset $M$ of a finite group $G$, we denote by $\exp(M)$ the least common multiple of the orders of the elements of $M$. Finally, for a nonabelian finite simple group $S$ and a word $w\in\F(X_1,\ldots,X_d)$, a \emph{coset word equation with respect to $w$ over $S$} is an equation of the form $w(s_1\alpha_1,\ldots,s_d\alpha_d)=\beta$ where $\alpha_1,\ldots,\alpha_d,\beta$ are fixed automorphisms of $S$, and $s_1,\ldots,s_d$ are variables ranging over $S$ (so that the solution set of such an equation is always a subset of $S^d$).

\section{Permutation groups}\label{sec3}

Some of our proofs require the study of the support of permutations
and its distributions in finite permutation groups. In this section we
obtain results in this direction, which may be of some independent interest.

For a permutation group $P\leq\Sym(\Omega)$ and $\sigma \in P$ we let $\supp(\sigma)$
denote the number of points moved by $\sigma$ and $\supp(P)$ the number of points
moved by some element of $P$. We also let $\fix(\sigma)$ denote the number of fixed
points of $\sigma$, and $\deg(P):=|\Omega|$.

\begin{theorem}\label{supp} Let $P\leq\Sym(\Omega)$ be a permutation group (where $P$ and
$\Omega$ are not assumed to be finite).
Let $c$ be a positive integer, and suppose $\supp(\sigma) \le c$ for all $\sigma \in P$.
Then
\begin{enumerate}
\item$|P| \le c!$;
\item$\supp(P) \le 2(c-1)$ if $c>1$.
\end{enumerate}
\end{theorem}

We note that the bound in part (1) is best possible for all $c$ (take $P = \Sym(c)$ acting on $\Omega = \{1, \ldots , c \})$.

The bound in part (2) is also best possible at least when $c=2^k$ for some $k\in\IN^+$, by the following example: let $H_k\leq\IF_2^{2^k}$ be the $[2^k,k,2^{k-1}]_2$-Hadamard code (see e.g.~\cite[p. 248]{V96}). By its definition, it is clear that there is a unique coordinate where all elements of $H_k$ are $0$; we project $H_k$ onto the $2^k-1$ other coordinates, resulting in a subspace $\tilde{H}_k\leq\IF_2^{2^k-1}$ (which we regard as an
additive group) with the following properties:
\begin{itemize}
\item every nonzero element of $\tilde{H}_k$ has exactly $2^{k-1}$ nonzero entries (equal to $1$);
\item for each $i\in\{1,\ldots,2^k-1\}$, there is an element of $\tilde{H}_k$ having entry $1$ in the $i$-th coordinate.
\end{itemize}
Set $\Omega:=\{1,\ldots,2^k-1\}\times\{0,1\}$. Consider the function $f:\tilde{H}_k\rightarrow\Sym(\Omega)$ where $f(x_1,\ldots,x_{2^k-1})$ is the product of the transpositions $((i,0),(i,1))$ for those $i\in\{1,\ldots,2^k-1\}$ where $x_i=1$. Then $f$ is an injective group homomorphism, so
the image $P:=f[\tilde{H}_k]$ is actually a subgroup of $\Sym(\Omega)$, and it satisfies $\supp(P)=|\Omega|=2(2^k-1)$ and that all nontrivial elements of $P$ have support size exactly $2\cdot 2^{k-1}=2^k$.

We now prove Theorem \ref{supp}.

\begin{proof}
We first assume $\Omega$ is finite, and then deduce the result without this assumption.

Set $n = |\Omega|$.
We may assume $P$ has no orbits of size $1$ in its action on $\Omega$, since we may delete these orbits
from $\Omega$, thereby obtaining a subset $\Omega' = \supp(P)$, and regard $P$ as a permutation group on $\Omega'$.

Suppose $P$ has $t$ orbits on $\Omega$, of sizes $n_1, \ldots , n_t$.
Then
\[
|P| \le n_1! \cdots n_t!.
\]
Since $\supp(\sigma) \le c$ for all $\sigma \in P$, we have $\fix(\sigma) \ge n-c$ for all $\sigma \in P$.
Consider the random variable $X=\fix(\sigma)$, where $\sigma\in P$ is assumed to be chosen uniformly at random. Then, by the
Cauchy-Frobenius Lemma (\enquote{The Lemma that is not Burnside's}), $\E(X) = t$. This yields $t \ge n-c$. In fact, since $\fix(1) = n$ we have $t > n-c$, hence
\[
\sum_{i=1}^t (n_i-1) = n-t \le c-1.
\]
Since $n_i \ge 2$ we have $n_i \le 2(n_i-1)$, and so
\[
\supp(P) = \sum_{i=1}^t n_i \le 2 \sum_{i=1}^t (n_i-1) \le 2(c-1).
\]
This proves part (2).

To prove part (1) we claim that
\[
\prod_{i=1}^t n_i! \le (1+ \sum_{i=1}^t (n_i-1))!.
\]
We prove the claim by induction on $t$, the case $t=1$ being trivial.

Assuming the claim for $t-1$ we have
\[
\prod_{i=1}^{t-1} n_i! \le (1+ \sum_{i=1}^{t-1} (n_i-1))!.
\]
Set $d = 1+ \sum_{i=1}^{t-1} (n_i-1)$. Then,
since $d \ge 1$ we have
\[
n_t! \le (d+1)(d+2) \cdots (d+n_t-1).
\]
Hence
\[
\prod_{i=1}^t n_i! \le d! n_t! \le d!(d+1)(d+2) \cdots (d+n_t-1) = (d+n_t-1)! = (1+ \sum_{i=1}^{t} (n_i-1))!,
\]
proving the claim. We conclude that
\[
|G| \le \prod_{i=1}^t n_i! \le
(1+ \sum_{i=1}^t (n_i-1))! \le c!,
\]
proving part (1).

Suppose now $\Omega$ is infinite. Let $\Omega'$ be the support of $P$, as above. We claim that $\Omega'$ is finite,
hence, regarding $P$ as a permutation group on $\Omega'$, we reduce to the finite case.

To prove the claim, choose $\sigma_1 \in P$ and denote its support by $B_1$. If $B_1 = \Omega'$ then $\Omega'$ has
size at most $c$ and we are done. Otherwise there exists $\sigma_2 \in P$ with support $B_2$ which is not contained
in $B_1$. If $B_1 \cup B_2 = \Omega'$ we are done. Otherwise we proceed so that in step $i$ we choose $\sigma_i \in P$
with support $B_i$ which is not contained in $\Omega_{i-1} := \cup_{j=1}^{i-1} B_j$. Let $P_i \le P$ be the subgroup generated by
$\sigma_1, \ldots , \sigma_i$ and let $\Omega_i = \cup_{j=1}^i B_j$. Then $\Omega_i$ is finite (of size at most $ci$) and
$P_i \le \Sym(\Omega_i)$. By the finite case we have $|\Omega_i| = \supp(P_i) \le 2(c-1)$. Since the sequence $|\Omega_j|$
is increasing the process must stop, which means that, for some $i$, $\Omega' = \Omega_i$ is finite. This completes the proof.

\end{proof}

We now prove a result on permutation groups that will be used later.

Let $C\in\IN$, and let $P\leq\Sym(\Omega)$ be a permutation group. We denote by $\SB_C(P)$ the set of all $\sigma\in P$ whose support on $\Omega$ is of size at most $C$.

\begin{theorem}\label{permGroupTheo}
There is a function $f:\left(0,1\right]\times\IN\rightarrow\left[1,\infty\right)$ such that the following holds: Let $\rho\in\left(0,1\right]$, $C\in\IN$, and assume that $P\leq\Sym(\Omega)$ is a permutation group of finite degree such that $|\SB_C(P)|\geq\rho|P|$. Then
\[
|P|\leq f(\rho,C).
\]
Indeed, one may choose $f$ to be the following function:
\[
f(\rho,C):=(\lfloor\rho^{-1}+C+1\rfloor !)^{\lceil 8(C-\log{\rho})\rceil}
\]
\end{theorem}

\begin{proof}
This is clear if $C=0$, since then $\SB_C(P)=\{\id_{\Omega}\}$, whence $|\SB_C(P)|\geq\rho|P|$ is equivalent to $|P|\leq\rho^{-1}$, and
\[
\rho^{-1}\leq\lfloor\rho^{-1}+1\rfloor\leq\lfloor\rho^{-1}+1\rfloor !\leq(\lfloor\rho^{-1}+1\rfloor !)^{\lceil 8\log{\rho^{-1}}\rceil}.
\]
The assertion is also clear if $C\geq\deg(P)$. So we may henceforth assume that $1\leq C<\deg(P)$. We first show the following claim: \enquote{If $P$ is transitive, then $\deg(P)\leq\rho^{-1}+C$.}

To see that this claim holds true, consider the random variable $X=\fix(\sigma)$, where $\sigma\in P$ is assumed to be chosen uniformly at random. Then as noted in the proof of Theorem \ref{supp}, by the Cauchy-Frobenius Lemma, $\E(X)=1$.

Moreover, the Markov inequality (see for instance \cite[p. 265]{BT08}) shows that, for each positive integer $k$,
\[
{\bf P}(X\geq k)\leq\frac{\E(X)}{k}=\frac{1}{k}.
\]
Applied with $k:=\deg(P)-C$, this yields
\[
{\bf P}(\sigma\in\SB_C(P))\leq\frac{1}{\deg(P)-C},
\]
so that
\[
\rho\leq\frac{1}{\deg(P)-C},
\]
or equivalently, $\deg(P)\leq\rho^{-1}+C$. This concludes the proof of the above claim.

The claim yields in particular that the asserted upper bound on $|P|$ holds when $P$ is transitive. Let us now give an argument for general $P$. Let $\Omega=\Omega_1 \sqcup \cdots \sqcup \Omega_t$ be the partition of $\Omega$ into the orbits of $P$. For $i=1,\ldots,t$, denote by $P_i\leq\Sym(\Omega_i)$ the (transitive) image of $P$ under the restriction homomorphism $\pi_i:P\rightarrow\Sym(\Omega_i)$, $\sigma\mapsto\sigma_{\mid\Omega_i}$. Observe that $\pi_i[\SB_C(P)]\subseteq\SB_C(P_i)$, and so $|\SB_C(P_i)|\geq\rho|P_i|$ as well. Hence if, for any $i\in\{1,\ldots,t\}$, one has $|\Omega_i|>\rho^{-1}+C$, one gets a contradiction to the above claim. So we may assume that $|\Omega_i|\leq\rho^{-1}+C$ for each $i=1,\ldots,t$; in particular, $|P_i|\leq\lfloor\rho^{-1}+C\rfloor !$.

Aiming for a contradiction, assume now additionally that
\[
|P|>(\lfloor\rho^{-1}+C+1\rfloor !)^{\lceil 8(C-\log{\rho})\rceil}.
\]
Then
\[
\frac{|P|}{(\lfloor\rho^{-1}+C\rfloor !)^j}>1
\]
for $j=0,1,\ldots,\lceil 8(C-\log{\rho})\rceil$, allowing us to choose, for $s:=\lceil 8(C-\log{\rho})\rceil$, a length $s$ sequence $(i_1,\ldots,i_s)$ of pairwise distinct indices from $\{1,\ldots, t\}$ such that for each $j\in\{1,\ldots,s\}$,
\[
c_j:=|\pi_{i_j}[\ker(\pi_{i_1})\cap\cdots\cap\ker(\pi_{i_{j-1}})]|\geq2.
\]
What this means is that among all the elements of $P$, there occur $c_1\geq2$ distinct values in the $i_1$-th coordinate, and after fixing any of the $c_1$ many values in the $i_1$-th coordinate and considering only such elements of $P$, there still occur $c_2\geq2$ distinct values in the $i_2$-th coordinate, and after fixing both the $i_1$-th and $i_2$-th coordinate, there still occur $c_3\geq2$ distinct values in the $i_3$-th coordinate, and so on.

Now consider $\pi:P\rightarrow\prod_{j=1}^s{P_{i_j}}$, the projection of $P$ to the coordinates number $i_1,\ldots,i_s$. The image $\pi[P]$ still satisfies that $|\SB_C(\pi[P])|\geq\rho|\pi[P]|$, but on the other hand,
\begin{align*}
\SB_C(\pi[P])\subseteq &\{(\sigma_{i_1},\ldots,\sigma_{i_s})\in \pi[P]\leq\prod_{j=1}^s{P_{i_j}} \mid  \\
&\exists M\subseteq\{1,\ldots,s\}:(|M|=C\text{ and }\sigma_{i_r}=\id_{\Omega_r}\text{ for all }r\in\{1,\ldots,s\}\setminus M)\}.
\end{align*}
Letting $d_1\geq d_2\geq\cdots\geq d_s$ be such that the multisets $\{c_1,\ldots,c_s\}$ and $\{d_1,\ldots,d_s\}$ are equal, this yields the following upper bound on the proportion of elements in $\pi[P]$ with support size at most $C$:
\[
\frac{1}{|\pi[P]|}|\SB_C(\pi[P])|\leq\frac{1}{d_1\cdots d_s}{s \choose C}d_1\cdots d_{C}=\frac{{s \choose C}}{d_{C+1}\cdots d_s}\leq\frac{(\frac{\e s}{C})^{C}}{2^{s-C}}=\frac{(\frac{2\e s}{C})^{C}}{2^s}.
\]
We thus get the desired contradiction if we can argue that
\begin{equation}\label{sbEq}
\frac{(\frac{2\e s}{C})^{C}}{2^s}<\rho.
\end{equation}
Recall that $s=\lceil 8(C-\log{\rho})\rceil$, and set $s':=\frac{s}{C}$, so that $s=C\cdot s'$. Then
\[
\frac{(\frac{2\e s}{C})^{C}}{2^s}=\frac{(2\e s')^{C}}{2^{s'C}}=(\frac{2\e s'}{2^{s'}})^{C}=(\frac{\e s'}{2^{s'-1}})^{C},
\]
and that last expression is strictly smaller than $\rho$ if and only if
\[
s'-\log_2{s'}-1>\frac{1-\frac{1}{C}\log{\rho}}{\log{2}}.
\]
Now by definition,
\[
s'=\frac{s}{C}=\frac{\lceil 8(C-\log{\rho})\rceil}{C}\geq 8(1-\frac{\log{\rho}}{C})\geq 8,
\]
and so
\[
s'-\log_2{s'}-1\geq\frac{1}{2}s'.
\]
Hence Formula (\ref{sbEq}) is implied by
\[
s'> 2\cdot\frac{1-\frac{1}{C}\log{\rho}}{\log{2}}=\frac{2}{\log{2}}(1-\frac{1}{C}\log{\rho}),
\]
which is clear by definition of $s'$.
\end{proof}

In various cases we can obtain better bounds on $|P|$ also for intransitive groups.
Let $t$ denote the number of orbits of $P \le \Sym(n)$, and let $r$ denote the rank of $P$ (namely the number of orbits on ordered pairs of points). Clearly $r \ge t^2$.

\begin{proposition}\label{betterbound}  With the above notation we have:
\begin{enumerate}
\item The probability that a random element $\sigma \in P$ satisfies $\supp(\sigma) > (1-\epsilon)n$ is at least $1 - t/(\epsilon n)$
for any $0< \epsilon <1$. Thus this probability tends to $1$ as $t = o(n)$.
\item The probability that a random element $\sigma \in P$ satisfies $\supp(\sigma) > (1-\epsilon)n - t$ is at least $1 - (r-t^2)/(\epsilon^2 n^2)$
for any $0 < \epsilon < 1$. Thus this probability tends to $1$ as $r-t^2 = o(n^2)$.
\end{enumerate}
\end{proposition}

\begin{proof}
The Markov inequality applied in the proof of the above theorem shows that,
for any fixed $\epsilon > 0$ we obtain (substituting $k = \epsilon n$),
\[
{\bf P}(\supp(\sigma) > (1-\epsilon)n) \ge 1 - t/k = 1 - \frac{t}{\epsilon n},
\]
which tends to $1$ provided $t = o(n)$. Part (1) follows.

For part (2) we use the second moment method for the random variable
$X = \fix(\sigma)$ ($\sigma \in P$). Then $\E(X) = t$, and as is well-known, by applying the Cauchy-Frobenius Lemma to the action of $P$ on $\{1,\ldots,n\}^2$, one also gets $\E(X^2) = r$. Therefore
\[
\Var(X) = \E(X^2)-\E(X)^2 = r - t^2.
\]

By the Chebyshev inequality (see for instance \cite[p. 267]{BT08}) we have
\[
{\bf P}(|X - \E(X)| \ge k) \le \frac{\Var(X)}{k^2}.
\]
Writing $k = \epsilon n$ we obtain
\[
{\bf P}(|X - t| < \epsilon n) \ge  1 - \frac{r-t^2}{\epsilon^2n^2}.
\]
Clearly $|X - t| < \epsilon n$ implies $\fix(\sigma) < t +  \epsilon n$, which yields
\[
\supp(\sigma) = n - \fix(\sigma) > (1-\epsilon)n - t.
\]
The result follows.
\end{proof}

Note that statement (1) of Proposition \ref{betterbound} implies that $\deg(P)\leq t\cdot{\bf P}(\supp(\sigma)\leq C)^{-1}+C$, which, adopting the notation from Theorem \ref{permGroupTheo} yields that $\deg(P)\leq t\rho^{-1}+C$, and so
\[
|P| \le \lfloor t\rho^{-1}+C \rfloor !,
\]
Similarly, statement (2) of Proposition \ref{betterbound} implies that, with the above notation, we have
$\deg(P)\leq \sqrt{r-t^2}\rho^{-1} + t + C$, which yields
\[
|P| \le \lfloor \sqrt{r-t^2}\rho^{-1} + t + C  \rfloor !.
\]

We conclude this section with the following example, which shows that (in the notation used in Proposition \ref{betterbound}(2)) $r-t^2$ is not always in $o(n^2)$:

\begin{example}\label{dihedralEx}
Let $P=\D_6=\Sym(3)$ in its regular action on itself (hence on $6$ points). Then $P$ is sharply $1$-transitive. For $m\in\IN^+$, $m\geq2$, let $\G_m$ be the set of length $m$ sequences $\vec{\sigma}=(\sigma_1,\ldots,\sigma_m)\in P^m$ such that $\ord(\sigma_i)=2$ for $i=1,\ldots,m$ and $|\{\sigma_1,\ldots,\sigma_m\}|\geq2$. Note that each such sequence $\vec{\sigma}$ is a generating sequence for $P$. Set $k_m:=|\G_m|$, denote by $\pi^{(m)}_i$, for $i=1,\ldots,m$, the projection $P^m\rightarrow P$ to the $i$-th coordinate, and let
\[
P_m:=\langle(\pi^{(m)}_1(\vec{\sigma}))_{\vec{\sigma}\in\G_m},\ldots,(\pi^{(m)}_m(\vec{\sigma}))_{\vec{\sigma}\in\G_m}\rangle\leq P^{\G_m}\cong P^{k_m}.
\]
Then $P_m$ is a $k_m$-fold subdirect power of $P$; in particular, all orbits $\Omega_j$, for $j=1,\ldots,k_m$, of $P_m$ are of length $6$. Note also that the $m$ listed generators of $P_m$ are pairwise distinct, so that $|P_m|\geq m$. For each $j\in\{1,\ldots,k_m\}$ and each $\omega\in\Omega_j$, the point stabilizer $(P_m)_{\omega}$ consists only of even length products of the listed generators of $P_m$; in particular, for each $l\in\{1,\ldots,k_m\}$, the restriction of each element of $(P_m)_{\omega}$ to $\Omega_l$ is contained in the unique index $2$ subgroup of the corresponding (sharply $1$-transitive) action of $P=\D_6$ on $\Omega_l$. Hence $(P_m)_{\omega}$ is intransitive on each orbit $\Omega_l$ of $P_m$, whence each Cartesian product $\Omega_j\times\Omega_l$ of orbits of $P$ splits into at least two distinct orbits under the component-wise action of $P^2$. In particular, $r(P_m)\geq 2t(P_m)^2$, and so
\[
r(P_m)-t(P_m)^2\geq t(P_m)^2=k_m^2=(\frac{6k_m}{6})^2=(\frac{\deg(P_m)}{6})^2=\frac{1}{36}\deg(P_m)^2.
\]
\end{example}

\section{Probabilistic identities}\label{sec4}

\subsection{Permutation-part-bounding words}\label{subsec4P1}

We now introduce another word property that will be relevant for the proof of Theorem \ref{mainTheo1}:

\begin{deffinition}\label{permDef}
Consider the following notations and concepts:
\begin{enumerate}
\item Let $H$ be a nontrivial finite semisimple group, say
\begin{align*}
&S_1^{n_1}\times\cdots\times S_r^{n_r}\leq H\leq\Aut(S_1^{n_1}\times\cdots\times S_r^{n_r}) \\
&=(\Aut(S_1)\wr\Sym(n_1))\times\cdots\times(\Aut(S_r)\wr\Sym(n_r)),
\end{align*}
where $S_1,\ldots,S_r$ are pairwise nonisomorphic nonabelian finite simple groups and $n_1,\ldots,n_r\in\IN^+$. For $i=1,\ldots,r$, denote by $\pi_i:H\rightarrow\Aut(S_i)\wr\Sym(n_i)$ the projection to the $i$-th coordinate, and let $H_i$ be the image of $H$ under $\pi_i$, which is again semisimple, with socle $S_i^{n_i}$. We introduce the following notations for isomorphism invariants of $H$:
\begin{enumerate}
\item $P(H):=H/(H\cap(\Aut(S_1)^{n_1}\times\cdots\times\Aut(S_r^{n_r})))$ for the so-called \emph{permutation part of $H$}, which we can view naturally as a subgroup of $\Sym(n_1)\times\cdots\times\Sym(n_r)$.
\item $\Perm(H)$ for the multiset $\{P(H_1),\ldots,P(H_r)\}$, and
\item $\MPO(H)$ for the number $\max\{|P(H_i)|\mid i=1,\ldots,r\}$.
\end{enumerate}
\item Let $w\in\F(X_1,\ldots,X_d)$ be a nontrivial word. We say that $w$ is \emph{permutation-part-bounding} (or \emph{PPB} for short) if and only if there is a function $f_w:\left(0,1\right]\rightarrow\left[1,\infty\right)$ such that for every nonsolvable finite group $G$ satisfying a probabilistic identity with respect to $w$ and $\rho$, $\MPO(H_1(G))\leq f_w(\rho)$.
\end{enumerate}
\end{deffinition}

Clearly, MB words are PPB. Moreover, we have the following:

\begin{lemmma}\label{ppbNlbLem}
The following hold:
\begin{enumerate}
\item Let $G$ be a finite group. Then $H_2(G)$ is a section of $\prod_{P\in\Perm(H_1(G))}{P}$.
\item PPB words are NLB.
\end{enumerate}
\end{lemmma}

\begin{proof}
For (1): By definition,
\[
H_2(G)=G_2(G)/R_2(G)=(H_1(G)/\Soc(H_1(G)))/\Rad(H_1(G)/\Soc(H_1(G))).
\]
It is thus sufficient to show that $G_2(G)=H_1(G)/\Soc(H_1(G))$ has a solvable normal subgroup $N$ such that $G_2(G)/N$ is isomorphic to a subgroup of $\prod_{P\in\Perm(H_1(G))}{P}$. Letting $\Soc(H_1(G))\cong S_1^{n_1}\times\cdots\times S_r^{n_r}$ where $S_1,\ldots,S_r$ are pairwise nonisomorphic nonabelian finite simple groups and $n_1,\ldots,n_r\in\IN^+$, we may view, up to natural isomorphism,
\[
S_1^{n_1}\times\cdots\times S_r^{n_r}\leq H_1(G)\leq\Aut(S_1^{n_1}\times\cdots\times S_r^{n_r}).
\]
We then find that
\begin{align*}
&N:=((\Aut(S_1)^{n_1}\times\cdots\times\Aut(S_r)^{n_r})\cap H_1(G))/(S_1^{n_1}\times\cdots\times S_r^{n_r}) \\
&\hookrightarrow\Out(S_1)^{n_1}\times\cdots\times\Out(S_r)^{n_r}
\end{align*}
is a suitable choice.

For (2): Let $w\in\F(X_1,\ldots,X_d)\setminus\{1\}$ be PPB, and assume that $G$ is a finite group that satisfies a probabilistic identity with respect to that word $w$ and some given $\rho\in\left(0,1\right]$. We want to bound $\lambda(G)$ in terms of $w$ and $\rho$. If $G$ is solvable, then $\lambda(G)=0$, so assume that $G$ is nonsolvable. Then $|\MPO(G/\Rad(G))|=|\MPO(H_1(G))|\leq f_w(\rho)$, where $f_w$ is as in the definition of PPB words. In other words, $|P|\leq f_w(\rho)$ for each $P\in\Perm(G/\Rad(G))$. Moreover, by \cite[Theorem 5.2]{LS18a}, there is an $N_w(\rho)>0$ such that all nonabelian composition factors of $G$ have order at most $N_w(\rho)$. In particular, the number of nonisomorphic simple direct factors in $\Soc(G/\Rad(G))$ is bounded from above by $N_w(\rho)$ (because for each $k\geq 1$, the number of isomorphism types of nonabelian finite simple groups up to order $k$ is at most $k$, as the orders of nonabelian finite simple groups are even and for each given order, there are at most two nonisomorphic nonabelian finite simple groups of that order). Using statement (1), it follows that
\[
60^{\lambda(G)-1}\leq|H_2(G)|\leq|\prod_{P\in\Perm(G/\Rad(G))}{P}|\leq f_w(\rho)^{N_w(\rho)},
\]
and thus
\[
\lambda(G)\leq \frac{1}{\log{60}}N_w(\rho)\log{f_w(\rho)}+1.
\]
\end{proof}

In particular, the proof of Theorem \ref{mainTheo1} is now reduced to the following, which we will show next:

\begin{lemmma}\label{avsmbIsPpbLem}
Almost VSMB words are PPB.
\end{lemmma}

\begin{proof}
Let $w$ be an almost VSMB word, let $\rho\in\left(0,1\right]$, and assume that a finite nonsolvable group $G$ satisfies a probabilistic identity with respect to $w$ and $\rho$. Then every quotient of $G$ also satisfies a probabilistic identity with respect to $w$ and $\rho$; in particular, writing $\Soc(H_1(G))=S_1^{n_1}\times\cdots\times S_r^{n_r}$ where $S_1,\ldots,S_r$ are pairwise nonisomorphic nonabelian finite simple groups and $n_1,\ldots,n_r\in\IN^+$, for $i=1,\ldots,r$, the group $H_{1,i}(G)$, defined as the projection of $H_1(G)\leq\Aut(S_1^{n_1})\times\cdots\times\Aut(S_r^{n_r})$ to the $i$-th coordinate, satisfies a probabilistic identity with respect to $w$ and $\rho$. Note that up to isomorphism, $S_i^{n_i}\leq H_{1,i}(G)\leq\Aut(S_i^{n_i})=\Aut(S_i)\wr\Sym(n_i)$, and that when setting
\[
P_{1,i}(G):=P(H_{1,i}(G))\hookrightarrow\Sym(n_i),
\]
one has by definition that $\Perm(H_1(G))=\{P_{1,i}(G),\ldots,P_{1,r}(G)\}$. So our goal is to find an upper bound in terms of $w$ and $\rho$ on $\max\{|P_{1,i}(G)|\mid i=1,\ldots,r\}$.

To that end, fix $i\in\{1,\ldots,r\}$, and for notational simplicity, write $S$ instead of $S_i$, $n$ instead of $n_i$, $H$ instead of $H_{1,i}(G)$, and $P$ instead of $P_{1,i}(G)$. For $\sigma\in P$, denote by $\Supp(\sigma)$ the set of points moved by $\sigma$ (so that, using the notation from Section \ref{sec3}, $\supp(\sigma)=|\Supp(\sigma)|$). Recall from above that $H$ satisfies a probabilistic identity with respect to $w$ and $\rho$, so we can fix an element $h=(\beta_1,\ldots,\beta_n)\psi\in H$ such that $p_{w,H}(h)\geq\rho$. Note: If $w$ is a \emph{repetition-free word}, i.e., if the maximum multiplicity of a variable in $w$ is $1$ (no variable occurs more than once in $w$), then the probabilistic identity implies that $|H|\leq\rho^{-1}$; in particular, $|P|\leq\rho^{-1}$ then, and we are done. So we may assume that $w$ is \emph{not} repetition-free. Writing $w=x_1^{\epsilon_1}\cdots x_{\ell}^{\epsilon_{\ell}}$ where $\ell$ is the length of $w$, $\epsilon_1,\ldots,\epsilon_{\ell}\in\{\pm1\}$ and $x_1,\ldots,x_{\ell}\in\{X_1,\ldots,X_d\}$, we can find indices $j_1,j_2\in\{1,\ldots,\ell\}$ with $j_1<j_2$ such that $x_{j_1}=x_{j_2}$, $x_j\not=x_{j_1}$ for all $j\in\{j_1+1,\ldots,j_2-1\}$, and the (possibly empty) word segment $x_{j_1+1}^{\epsilon_{j_1+1}}\cdots x_{j_2-1}^{\epsilon_{j_2-1}}$ is repetition-free. Moreover, for $j=1,\ldots,\ell$, define the word
\[
v_j:=\begin{cases}x_1^{\epsilon_1}\cdots x_{j-1}^{\epsilon_{j-1}}, & \text{if }\epsilon_j=1, \\ x_1^{\epsilon_1}\cdots x_j^{\epsilon_j} & \text{if }\epsilon_j=-1,\end{cases}
\]
see also \cite[Lemma 2.7]{Bor17a} and our Notation \ref{segmentsNot}(1), and set $v:=v_{j_1}^{-1}v_{j_2}$, see also Notation \ref{segmentsNot}(2). Note that by choice of $j_1$ and $j_2$, $v$ is a nonempty reduced word in which some variable occurs with multiplicity $1$.

We bound the number of solutions to the equation $w(y_1,\ldots,y_d)=h$, where $y_1,\ldots,y_d$ are variables ranging over $H$, in a $\Soc(H)$-coset-wise counting argument. More precisely, fix first a $d$-tuple $(\sigma_1,\ldots,\sigma_d)\in P^d$. There are two fundamentally different cases in the counting argument, according to whether or not $v(\sigma_1,\ldots,\sigma_d)\in\SB_{C(\rho)}(P)$, where $C(\rho):=\ell^2\cdot\frac{\log(2/\rho)}{\log(1+1/(N_w(\rho)^{\ell}-1))}$ and $N_w(\rho)$ is chosen such that all nonabelian composition factors of a finite group that satisfies a probabilistic identity with respect to $w$ and $\rho$ have order at most $N_w(\rho)$.
\begin{enumerate}
\item Assume first that $v(\sigma_1,\ldots,\sigma_d)\notin\SB_{C(\rho)}(P)$, i.e., that $\supp(v(\sigma_1,\ldots,\sigma_d))>C(\rho)$. For each $k=1,\ldots,d$, fix one of the $[(H\cap\Aut(S)^n):S^n]$ many cosets of $S^n$ in $H$ that have permutation part $\sigma_k$, say with coset representative $(\alpha_{k,1},\ldots,\alpha_{k,n})\sigma_k$, and consider the equation
\[
w((s_{1,1}\alpha_{1,1},\ldots,s_{1,n}\alpha_{1,n})\sigma_1,\ldots,(s_{d,1}\alpha_{d,1},\ldots,s_{d,n}\alpha_{d,n})\sigma_d)=h=(\beta_1,\ldots,\beta_n)\psi,
\]
where the $s_{k,t}$, for $k=1,\ldots,d$ and $t=1,\ldots,n$, are variables ranging over $S$. As described in \cite[Lemma 2.7]{Bor17a}, this equation can be rewritten into the conjunction of the single word equation $w(\sigma_1,\ldots,\sigma_d)=\psi$ and the system of coset word equations over $S$ with respect to some variations of $w$ whose $t$-th equation, for $t\in\{1,\ldots,n\}$, looks like this:
\[
(s_{\iota(1),\chi_1^{-1}(t)}\alpha_{\iota(1),\chi_1^{-1}(t)})^{\epsilon_1}\cdots(s_{\iota(\ell),\chi_{\ell}^{-1}(t)}\alpha_{\iota(\ell),\chi_{\ell}^{-1}(t)})^{\epsilon_{\ell}}=\beta_t,
\]
where $\iota$ is the unique function $\{1,\ldots,\ell\}\rightarrow\{1,\ldots,d\}$ such that $x_j=X_{\iota(j)}$ for $j=1,\ldots,\ell$, and $\chi_j=v_j(\sigma_1,\ldots,\sigma_d)$ for $j=1,\ldots,\ell$.

Hence for each $t\in\Supp(v(\sigma_1,\ldots,\sigma_d))$, the underlying word of the $\chi_{j_1}(t)$-th coset word equation in the above equation system is a proper variation of $w$, as follows by considering the $j_1$-th and $j_2$-th factors in the product on the left-hand side: $\iota(j_1)=\iota(j_2)$ (i.e., $w$ has the same variable, possibly with different exponents $\pm1$, in those positions), but
\[
\chi_{j_1}^{-1}(\chi_{j_1}(t))=t\not=(\chi_{j_2}^{-1}\chi_{j_1})(t)=\chi_{j_2}^{-1}(\chi_{j_1}(t))
\]
(so the second indices of the variables at those positions in the $\chi_{j_1}(t)$-th coset word equation are different). As $w$ is assumed to be almost VSMB, this implies that each coset word equation labeled by an index from $\chi_{j_1}[\Supp(v(\sigma_1,\ldots,\sigma_d))]$ is not universally solvable; in particular, since $|S|\leq N_w(\rho)$, its proportion of solutions (among the variables that occur in it) is at most $1-\frac{1}{N_w(\rho)^l}$.

But as in \cite[proof of Lemma 2.12]{Bor17a}, since $|\chi_{j_1}[\Supp(v(\sigma_1,\ldots,\sigma_d))]|>C(\rho)$, we can find at least $\lceil C(\rho)/\ell^2\rceil$ pairwise distinct indices in $\chi_{j_1}[\Supp(v(\sigma_1,\ldots,\sigma_d))]$ such that the corresponding equations in the above system have pairwise disjoint occurring variable sets (i.e., they are \enquote{pairwise independent}), and this implies that the proportion of solutions (in $S^{nd}$) of the entire system of equations is at most
\[
(1-\frac{1}{N_w(\rho)^{\ell}})^{\lceil C(\rho)/\ell^2\rceil}\leq(1-\frac{1}{N_w(\rho)^{\ell}})^{C(\rho)/\ell^2}=\frac{\rho}{2},
\]
where the equality is by definition of $C(\rho)$.
\item Assume now that $v(\sigma_1,\ldots,\sigma_d)\in\SB_{C(\rho)}(P)$. Then we do not give a nontrivial upper bound on the number of solutions per $d$-tuple of socle cosets with permutation parts $(\sigma_1,\ldots,\sigma_d)$, but we note that since $v$ contains some variable with multiplicity $1$, the proportion of such $d$-tuples $(\sigma_1,\ldots,\sigma_d)$ in $P^d$ is exactly $\frac{1}{|P|}|\SB_{C(\rho)}(P)|$.
\end{enumerate}
In combination, this yields the following:
\[
\rho\leq p_{w,H}(h)\leq\frac{\rho}{2}+\frac{1}{|P|}|\SB_{C(\rho)}(P)|,
\]
and thus
\[
\frac{1}{|P|}|\SB_{C(\rho)}(P)|\geq\frac{\rho}{2},
\]
so that an application of Theorem \ref{permGroupTheo} shows that $|P|$ can indeed be bounded from above in terms of $w$ and $\rho$, as required.
\end{proof}

\subsection{Proof of Corollary \ref{mbRadCor}}\label{subsec4P2}

Let $G$ be a finite group. Assume first that $[G:\Rad(G)]\leq C$ for some constant $C>0$. Then since $\Rad(G)$ is solvable (i.e., it only has abelian composition factors), for each nonabelian finite simple group $S$, the multiplicities of $S$ in $G$ and $G/\Rad(G)$ are the same. It follows that $[G:\Rad(G)]$, and hence $C$, is an upper bound on the product of the orders of the nonabelian composition factors of $G$, counted with multiplicities. In particular, the maximum multiplicity of a nonabelian composition factor of $G$ is at most $\log_{60}(C)$. This shows the implication \enquote{$\Leftarrow$} in the first sentence of Corollary \ref{mbRadCor}.

Now assume that for each nonabelian finite simple group $S$, the multiplicity of $S$ in $G$ is at most $C_S$ for some constant $C_S>0$ that may depend on $S$. Assume also that the maximum order of a nonabelian composition factor of $G$ is bounded from above by another constant $C>0$. Then let $D$ be the maximum value of $C_S$ where $S$ ranges over the (finitely many) nonabelian finite simple groups of order at most $C$, so that any nonabelian composition factor of $G$ occurs with multiplicity at most $D$. It follows that the socle $T_1(G)$ of $G/\Rad(G)$, which is of the form $S_1^{n_1}\times\cdots\times S_r^{n_r}$ where $S_1,\ldots,S_r$ are pairwise nonisomorphic nonabelian finite simple groups and $n_1,\ldots,n_r\in\IN^+$, satisfies
\[
|T_1(G)|\leq C^{Dr}\leq C^{CD},
\]
where the latter inequality uses that there are at most $C$ distinct isomorphism types of nonabelian finite simple groups of order at most $C$ (as was already observed in the proof of Lemma \ref{ppbNlbLem}(2) above). This concludes the proof of the implication \enquote{$\Rightarrow$} in the first sentence of Corollary \ref{mbRadCor}.

For the second sentence (the \enquote{In particular}), just observe that $[G:\Rad(G)]\geq\prod_{k=1}^{\infty}{|T_k(G)|}\geq 60^{\lambda(G)}$. This concludes the proof of Corollary \ref{mbRadCor}.

We thus have the following implication diagram between the various word properties considered in this paper:

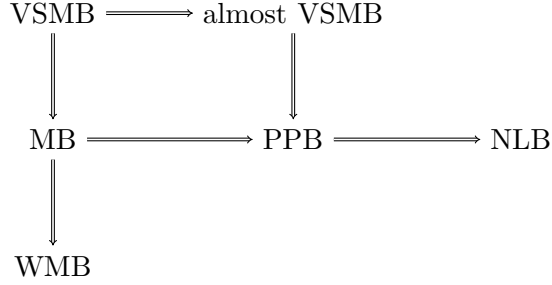
\begin{center}
\begin{tikzpicture}
\matrix (m) [matrix of math nodes, row sep=3em,
column sep=3em]
{ \text{VSMB} & \text{almost VSMB} & \\
\text{MB} & \text{PPB} & \text{NLB} \\
\text{WMB} & & \\
};
\path[-implies]
(m-1-1) edge[double] (m-1-2)
(m-1-1) edge[double] (m-2-1)
(m-1-2) edge[double] (m-2-2)
(m-2-1) edge[double] (m-2-2)
(m-2-1) edge[double] (m-3-1)
(m-2-2) edge[double] (m-2-3);
\end{tikzpicture}
\end{center}

\subsection{Proofs of Theorem \ref{mainTheo1} and Corollary \ref{mainCor1}}\label{subsec4P3}

The proof of Theorem \ref{mainTheo1} is immediate by combining Lemmas \ref{ppbNlbLem} and \ref{avsmbIsPpbLem}. For Corollary \ref{mainCor1}, note that by \cite[Theorem 1.1.2(3)]{Bor17a}, all nontrivial words $w$ of length at most $8$ are almost VSMB, so that we can conclude by an application of Theorem \ref{mainTheo1}.

\section{Identities}\label{sec5}

\subsection{Segment identities}\label{subsec5P1}

As noted in the Introduction, we will prove a result (Theorem \ref{mainTheo2} below) which will allow us to show that under certain assumptions, if a finite semisimple group $H$ satisfies some identity $w\equiv1$, then the permutation part $P(H)$ satisfies a shorter identity $v\equiv1$, where $v$ is some proper segment of $w$. Let us first introduce some notations and terminology and then formulate and prove Theorem \ref{mainTheo2}.

\begin{nottation}\label{segmentsNot}
Let $w\in\F(X_1,\ldots,X_d)$, say $w=x_1^{\epsilon_1}\cdots x_{\ell}^{\epsilon_{\ell}}$ where $\ell$ is the length of $w$, and for $i=1,\ldots,\ell$, $x_i\in\{X_1,\ldots,X_d\}$ and $\epsilon_i\in\{\pm1\}$.
\begin{enumerate}
\item For $i=1,\ldots,\ell$, set
\[
I_i(w):=\begin{cases}x_1^{\epsilon_1}\cdots x_{i-1}^{\epsilon_{i-1}}, & \text{if }\epsilon_i=1, \\ x_1^{\epsilon_1}\cdots x_i^{\epsilon_i}, & \text{if }\epsilon_i=-1.\end{cases}
\]
\item For $1\leq i<j\leq \ell$, set
\[
\Delta_{i,j}(w):=I_i(w)^{-1}I_j(w)
\]
\end{enumerate}
\end{nottation}

Note the following two simple facts:

\begin{remmark}\label{segmentsRem}
Using the notation from Notation \ref{segmentsNot}, we note the following:
\begin{enumerate}
\item The words $\Delta_{i,j}(w)$ are segments of $w$.
\item $\Delta_{i,j}(w)$ is empty if and only if $j=i+1$, $\epsilon_i=-1$ and $\epsilon_j=\epsilon_{i+1}=1$. In particular, since $w$ is reduced, $\Delta_{i,j}(w)$ is always nonempty if $i$ and $j$ are such that $x_i=x_j$.
\item $\Delta_{i,j}(w)=w$ if and only if $i=1$, $j=\ell$, $\epsilon_1=1$ and $\epsilon_{\ell}=-1$.
\end{enumerate}
\end{remmark}

\begin{deffinition}\label{splitDef}
Let $w\in\F(X_1,\ldots,X_d)$, with notation as in Notation \ref{segmentsNot}. Moreover, let $w'=y_1^{\epsilon_1}\cdots y_{\ell}^{\epsilon_{\ell}}$ be a variation of $w$, and let $1\leq i<j\leq\ell$. We say that $w'$ is an \emph{$(i,j)$-split variation of $w$} if and only if $x_i=x_j$ and $y_i\not=y_j$.
\end{deffinition}

\begin{theoremm}\label{mainTheo2}
Let $w\in\F(X_1,\ldots,X_d)$, with notation as in Notation \ref{segmentsNot}. Also, assume that for some given $i,j\in\{1,\ldots,\ell\}$ with $i<j$ and $x_i=x_j$, all $(i,j)$-split variations of $w$ are WMB. Then, if a finite semisimple group $H$ satisfies the identity $w\equiv1$, then the permutation part $P(H)$ satisfies the identity $\Delta_{i,j}(w)\equiv 1$. In particular, there is a nontrivial word $v\in\F(X_1,\ldots,X_d)$ of length strictly smaller than $\ell$ such that $P(H)$ satisfies the identity $v\equiv1$.
\end{theoremm}

\begin{proof}
The \enquote{In particular} follows from the main statement, as by Remark \ref{segmentsRem}(1,2), $\Delta_{i,j}(w)$ is a nonempty segment of $w$, and so usually, one will simply choose $v:=\Delta_{i,j}(w)$, unless $\Delta_{i,j}(w)=w$, which by Remark \ref{segmentsRem}(3) can only happen if $w=xvx^{-1}$ with $v\in\F(X_1,\ldots,X_d)\setminus\{1\}$ is not cyclically reduced, in which case $H$ and thus $P(H)$ satisfies the identity $v\equiv 1$. We thus focus on the proof of the main statement now.

Say $\Soc(H)=S_1^{n_1}\times\cdots\times S_r^{n_r}$ where $S_1,\ldots,S_r$ are pairwise nonisomorphic nonabelian finite simple groups and $n_1,\ldots,n_r\in\IN^+$. Then $H$ is a subdirect product of semisimple groups $H_k$, $k=1,\ldots,r$, such that $\Soc(H_k)=S_k^{n_k}$ for each $k$, and such that $P(H)$ is a subdirect product of the permutation parts $P(H_k)$, for $k=1,\ldots,r$. Hence it suffices to show that each $P(H_k)$ satisfies the identity $\Delta_{i,j}(w)\equiv 1$. This shows that we may assume w.l.o.g.~that $\Soc(H)=S^n$ for some nonabelian finite simple group $S$ and some $n\in\IN^+$.

Aiming for a contradiction, we will also assume that $P(H)$ does \emph{not} satisfy $\Delta_{i,j}(w)\equiv 1$. Then we can fix $\sigma_1,\ldots,\sigma_d\in P(H)$ with $\Delta_{i,j}(w)(\sigma_1,\ldots,\sigma_d)\not=\id$. Moreover, we fix $m_0\in\{1,\ldots,n\}$ with $\Delta_{i,j}(w)(\sigma_1,\ldots,\sigma_d)(m_0)\not=m_0$, and set $m_1:=I_i(w)(\sigma_1,\ldots,\sigma_d)(m_0)$. Finally, we fix automorphism tuples $\vec{\alpha}_k=(\alpha_{k,1},\ldots,\alpha_{k,n})\in\Aut(S)^n$, for $k=1,\ldots,d$, such that $\vec{\alpha}_k\sigma_k\in H$.

By assumption, we have that $w_H(S^n\vec{\alpha}_1\sigma_1,\ldots,S^n\vec{\alpha}_d\sigma_d)=\{1_H\}$. In particular, letting $s_{k,m}$, for $k=1,\ldots,d$ and $m=1,\ldots,n$, be variables ranging over $S$, then by \cite[Lemma 2.7]{Bor17a}, we have that a certain system of $n$ coset word equations over $S$ in the variables $s_{k,m}$ is universally solvable, and setting $\chi_t:=I_t(w)(\sigma_1,\ldots,\sigma_d)$ for $t=1,\ldots,l$ and denoting by $\iota$ the unique function $\{1,\ldots,\ell\}\rightarrow\{1,\ldots,d\}$ such that for $t=1,\ldots,\ell$, $x_t=X_{\iota(t)}$, one of the equations from the system is
\begin{equation}\label{m1Eq}
(s_{\iota(1),\chi_1^{-1}(m_1)}\alpha_{\iota(1),\chi_1^{-1}(m_1)})^{\epsilon_1}\cdots(s_{\iota(\ell),\chi_{\ell}^{-1}(m_1)}\alpha_{\iota(\ell),\chi_{\ell}^{-1}(m_1)})^{\epsilon_{\ell}}=1.
\end{equation}
Note that by assumption, $\iota(i)=\iota(j)$, but also
\begin{align*}
\chi_j^{-1}(m_1)&=(I_j(w)(\sigma_1,\ldots,\sigma_d))^{-1}(m_1) \\
&=(\Delta_{i,j}(w)(\sigma_1,\ldots,\sigma_d)^{-1}\cdot I_i(w)(\sigma_1,\ldots,\sigma_d)^{-1})(m_1) \\
&=\Delta_{i,j}(w)(\sigma_1,\ldots,\sigma_d)^{-1}(m_0)\not=m_0=\chi_i^{-1}(m_1).
\end{align*}
Hence Equation (\ref{m1Eq}) is a universally solvable coset word equation over $S$ with respect to some $(i,j)$-split variation $w'$ of $w$. But by assumption, $S$ does not satisfy any coset identity with respect to $w'$, which is the desired contradiction.
\end{proof}

\subsection{A consequence of Theorem \ref{mainTheo2}}\label{subsec5P2}

Using Theorem \ref{mainTheo2}, we can show the following, which will be used in the proof of Corollary \ref{mainCor2}:

\begin{propposition}\label{segmentProp}
Let $w\in\F(X_1,\ldots,X_d)$, with notation as in Notation \ref{segmentsNot}. Also, assume that for some given $k\in\{1,\ldots,d\}$, $\mu_w(X_k)\leq3$. Finally, let $i,j\in\{1,\ldots,\ell\}$ with $i<j$ such that $x_i=x_j=X_k$. Then if a finite semisimple group $H$ satisfies the identity $w\equiv1$, $P(H)$ satisfies $\Delta_{i,j}(w)\equiv1$; in particular, $P(H)$ satisfies a nontrivial identity of length strictly shorter than $\ell$ then.
\end{propposition}

\begin{proof}
The proof of the \enquote{In particular} is as for Theorem \ref{mainTheo2}. For the main statement: Since $\mu_w(X_k)\leq 3<2\cdot 2$, in each $(i,j)$-split variation $w'$ of $w$, there will be a variable that occurs with multiplicity exactly $1$. Hence $w'$ is VSMB, in particular WMB, by \cite[Proposition 3.1(1)]{Bor17a}.
\end{proof}

\subsection{Proof of Corollary \ref{mainCor2}}\label{subsec5P3}

By \cite[Theorem 1.1.2(3)]{Bor17a} and Corollary \ref{mainCor1}, it suffices to consider words $w$ of lengths $9$, $10$ or $11$. We start with the length $9$ case. Then the existence of $L_w$ (actually, with $L_w=0$) is clear if $w$ is a power of single variable. So we may also assume that $w$ contains at least two distinct variables. But if the total number of variables occurring in $w$ is at least $3$, then since $9<3\cdot 4$, there is a variable occurring with multiplicity at most $3$ in $w$. Hence by Proposition \ref{segmentProp}, $P(H_1(G))$ satisfies an identity $v\equiv 1$ for some word $v$ of length at most $8$. By Corollary \ref{mainCor1}, $v$ is NLB, and so $P(H_1(G))$ satisfying $v\equiv1$ entails that $\lambda(P(H_1(G)))$ (and thus $\lambda(G)$) is bounded from above by some constant, as required.

So we may henceforth assume that $w=w(x,y)$ is a two-variable word, and moreover (by an argument as in the previous paragraph, using Proposition \ref{segmentProp}), we may assume that each variable that occurs in $w$ does so with multiplicity at least $4$. Since $9<2\cdot 5$, one of the two variables, say w.l.o.g.~$x$, occurs with multiplicity exactly $4$ in $w$. Using the notation of Notation \ref{segmentsNot} for $w$ (with $l=9$, of course), fix a pair $(i,j)$ with $1\leq i<j\leq 9$ and $x_i=x_j=x$.

We will now argue that each $(i,j)$-split variation $w'$ of $w$ is WMB. Since $\mu_w(x)=4<3\cdot 2$, at least one of the variables in $w'$ derived from $x$, say $x'$, must occur with multiplicity at most $2$. If $\mu_{w'}(x')=1$, $w'$ is VSMB, in particular WMB, by \cite[Proposition 3.1(1)]{Bor17a}. So assume that $\mu_{w'}(x')=2$. The segment between the two occurrences of $(x')^{\pm1}$ in $w'$ is of length at most $7$, and thus it is VSMB by \cite[Theorem 1.1.2(3)]{Bor17a}. In view of this and \cite[Proposition 3.1(2,3)]{Bor17a}, $w'$ is VSMB, in particular WMB.

An application of Theorem \ref{mainTheo2} now yields that $P(H_1(G))$ satisfies an identity of the form $v\equiv1$ where $v$ is a word of length at most $8$. Again, by Corollary \ref{mainCor1}, $v$ is NLB, and so $\lambda(P(H_1(G)))$ is bounded from above by some constant.

The arguments for words of length $\ell\in\{10,11\}$ are largely similar, so we only sketch them. The first paragraph of the above argument can almost literally be carried over, replacing $9$ by $\ell$, of course, and not only referring to Corollary \ref{mainCor1} at the end, but also to the cases of length $9$ resp.~lengths $9$ and $10$ already done by then. In the two-variable case $w=w(x,y)$ with $\mu_w(x),\mu_w(y)\geq4$, sine $\ell<2\cdot 6$, we get that one of the two variables, say w.l.o.g.~$x$, occurs with multiplicity $4$ or $5$ in $w$. When choosing the pair $(i,j)$ with $1\leq i<j\leq\ell$ with $x_i=x_j=x$, one must also choose it such that the difference $j-i$ is maximal among all such pairs. This way, in the third paragraph of the argument, it is ensured that the segment $s$ between the two occurrences of $x'$ in $w'$ is of length at most $\ell-3$ (not just $\ell-2$, as in the argument for length $9$ words). For $\ell=10$, one can then conclude as in the length $9$ case, and for $\ell=11$, one needs the additional observation that $s$ cannot be an $8$-th or $(-8)$-th power of a single variable, for then some variable (necessarily $y$) occurs in $w$ with multiplicity at least $8$, so that $\mu_w(x)\leq3$, a contradiction.

\section{Power words}\label{sec6}

\subsection{Identities}\label{subsec6P1}

It is clear by a result of Segal \cite[Theorem 10]{Seg18a} that for each positive integer $e$, if a finite group $G$ satisfies the identity $x^e\equiv1$ (in other words, if $\exp(G)\mid e$), then $\lambda(G)$ is bounded from above in terms of $e$ (actually, Segal's result says that the same holds true if $\lambda(G)$ is replaced by the generalized Fitting height of $G$, which is an upper bound on $\lambda(G)$). Now Segal's proof uses the following, which is based on \cite[proof of Theorem 4.4.1]{HH56a} and the Feit-Thompson theorem:

\begin{lemmma}\label{hhLem}
Let $x$ be a variable, let $e\in\IN^+$, and let $H$ be a nontrivial finite semisimple group satisfying the identity $x^e\equiv1$ (in particular, $e$ is even). Then $P(H)$ satisfies the identity $x^{e/2}\equiv1$.\qed
\end{lemmma}

The aim of this subsection is two-fold: Firstly, to show a slightly stronger variant of Lemma \ref{hhLem} (see Lemma \ref{gcdBadLem} below), and secondly, to use a Segal-like argument for gaining a simple explicit upper bound on the nonsolvable length $\lambda(G)$ in terms of $\exp(G)$ (see Proposition \ref{powerIdLambdaProp} below).

Let us start with the stronger version of Lemma \ref{hhLem}, for which we introduce the following:

\begin{deffinition}\label{goodBadDef}
Let $x$ be any fixed variable. Call a positive integer $e$ \emph{good} if and only if the word $x^e$ is MB, and otherwise, call $e$ \emph{bad}. Moreover, for fixed $e\in\IN^+$, denote by $\BAD(e)$ the set of all positive divisors of $e$ that are bad.
\end{deffinition}

\begin{lemmma}\label{gcdBadLem}
Let $x$ be a variable, let $e\in\IN^+$, and let $H$ be a nontrivial finite semisimple group satisfying the identity $x^e\equiv1$ (in particular, $e$ is bad). Then $P(H)$ satisfies the identity $x^{e/\gcd(\BAD(e))}\equiv1$.
\end{lemmma}

\begin{proof}
We may w.l.o.g.~assume that $S^n\leq H\leq\Aut(S^n)$ for some nonabelian finite simple group $S$ and some $n\in\IN^+$ (as $H$ is, in general, a subdirect product of such groups, and likewise, $P(H)$ is a subdirect product of the permutation parts of those groups). Fix $\sigma\in P(H)$. We will show that $\sigma$ can only have cycles of lengths of the form $\frac{e}{d}$ where $d\in\BAD(e)$, and once we will have shown that, we will be done, as this implies that $\ord(\sigma)\mid\lcm_{d\in\BAD(e)}{\frac{e}{d}}=\frac{e}{\gcd(\BAD(e))}$.

So let $\zeta=(i_1,\ldots,i_{\ell})$ be a length $\ell$ cycle of $\sigma$. Note firstly that $\ell\mid e$, since $P(H)$, being a quotient of $H$, also satisfies the identity $x^e\equiv1$. Now fix $(\alpha_1,\ldots,\alpha_n)\in\Aut(S)^n$ such that $(\alpha_1,\ldots,\alpha_n)\sigma\in H$. It follows that for all $s_1,\ldots,s_n\in S$,
\[
((s_1\alpha_1,\ldots,s_n\alpha_n)\sigma)^e=1,
\]
and the expression on the left-hand side can be written as an element of $\Aut(S)^n$ whose $i_{\ell}$-th entry is
\[
(s_{i_{\ell}}\alpha_{i_{\ell}}s_{i_{\ell-1}}\alpha_{i_{\ell-1}}\cdots s_{i_1}\alpha_{i_1})^{e/\ell},
\]
which must in particular also be $1$ for all choices of $s_{i_1},\ldots,s_{i_{\ell}}\in S$. This shows that $S$ satisfies a coset identity with respect to $x^{e/\ell}$, and so $e/\ell$ is bad by \cite[Proposition 2.9(3)]{Bor17a}, i.e., $\ell=\frac{e}{d}$ for some $d\in\BAD(e)$, as required.
\end{proof}

Note that by \cite[Corollary 5.2]{Bor17a}, all bad positive integers are even, and so in Lemma \ref{gcdBadLem}, $\gcd(\BAD(e))\geq2$, whence Lemma \ref{gcdBadLem} does imply Lemma \ref{hhLem}, as asserted above. While it is true that the greatest common divisor of all bad positive integers is $2$ (since, for example, $8$ and $18$ are bad by \cite[Theorem 1.1.2(1)]{Bor17a}), and thus that Lemma \ref{gcdBadLem} does not always provide strictly stronger information than Lemma \ref{hhLem}, in some cases, it is better. As a somewhat extreme example, note that $\BAD(30)=\{30\}$ by \cite[Theorem 1.1.2(1)]{Bor17a}, and so by Lemma \ref{gcdBadLem}, $H$ satisfying the identity $x^{30}\equiv1$ implies that $P(H)$ is trivial (as opposed to it just satisfying the identity $x^{15}\equiv1$, which is what Lemma \ref{hhLem} gives).

Using the bound from Lemma \ref{hhLem}, we will now show, similarly to Segal's proof of \cite[Theorem 10]{Seg18a}:

\begin{propposition}\label{powerIdLambdaProp}
For every finite group $G$, $\lambda(G)\leq\nu_2(\exp(G))$.
\end{propposition}

\begin{proof}
By induction on $v:=\nu_2(\exp(G))$. If $v=0$, then $G$ is solvable by the Feit-Thompson Theorem, so $\lambda(G)=0$, and the bound is clear in that case. Now assume that $v\geq1$, and also assume that $G$ is nonsolvable (otherwise, again, $\lambda(G)=0$ and the bound is clear). Then since $G$ satisfies the identity $x^{\exp(G)}\equiv1$, so does $H_1(G)=G/\Rad(G)$. By Lemma \ref{hhLem}, it follows that $P(H_1(G))$ satisfies the identity $x^{\exp(G)/2}\equiv 1$, and thus, by the induction hypothesis,
\[
\lambda(G)-1=\lambda(P(H_1(G)))\leq\nu_2(\exp(G)/2)=\nu_2(\exp(G))-1,
\]
which yields the desired bound $\lambda(G)\leq\nu_2(\exp(G))$.
\end{proof}

\subsection{Probabilistic identities}\label{subsec6P2}

In this subsection, we are concerned with the proof of Theorem \ref{newMBTheo}. It relies on the following two lemmas of some independent interest:

\begin{lemmma}\label{nonConstantLem}
Let $w\in\F(X_1,\ldots,X_d)$, let $S$ be a nonabelian finite simple group, and let $\alpha_1,\ldots,\alpha_d\in\Aut(S)$. The following are equivalent:
\begin{enumerate}
\item $w(S\alpha_1,\ldots,S\alpha_d)\not=\{w(\alpha_1,\ldots,\alpha_d)\}$.
\item $w(S\alpha_1,\ldots,S\alpha_d)\not=\{1\}$.
\end{enumerate}
\end{lemmma}

We note that by Lemma \ref{nonConstantLem}, \cite[Corollary 5.2]{Bor17a} may be viewed as a direct consequence of Nikolov's earlier result \cite[Proposition 10]{Nik16a}.

\begin{lemmma}\label{psl23fLem}
Let $f\in\IN^+$ be odd, let $S=\PSL_2(3^f)$, and let $\alpha$ be an automorphism of $S$ with nontrivial diagonal part and whose field part is of order $f$. Then $\exp(S\alpha)=4f$.
\end{lemmma}

Computer calculations show that the statement of Lemma \ref{psl23fLem} is also true for $f\in\{2,4\}$, so it might actually hold for all $f\in\IN^+$. Let us now prove these two lemmas before proceeding with the proof of Theorem \ref{newMBTheo}.

\begin{proof}[Proof of Lemma \ref{nonConstantLem}]
For \enquote{(1) $\Rightarrow$ (2)}: We will show the contraposition: Assume that $w(S\alpha_1,\ldots,S\alpha_d)=\{1\}$. Then, since $w(\alpha_1,\ldots,\alpha_d)\in w(S\alpha_1,\ldots,S\alpha_d)$, it follows that $w(\alpha_1,\ldots,\alpha_d)=1$, and so $w(S\alpha_1,\ldots,S\alpha_d)=\{1\}=\{w(\alpha_1,\ldots,\alpha_d)\}$, as required.

For \enquote{(2) $\Rightarrow$ (1)}: Let $\beta\in w(S\alpha_1,\ldots,S\alpha_d)\setminus\{1\}\subseteq\Aut(S)\setminus\{1\}$. Then there is an $s\in S$ with $\beta^s\not=\beta$. But
\[
\beta^s\in w(S\alpha_1,\ldots,S\alpha_d)^s=w((S\alpha_1)^s,\ldots,(S\alpha_d)^s)=w(S\alpha_1,\ldots,S\alpha_d).
\]
It follows that $|w(S\alpha_1,\ldots,S\alpha_d)|\geq2$. In particular, $w(S\alpha_1,\ldots,S\alpha_d)$ cannot be equal to the singleton $\{w(\alpha_1,\ldots,\alpha_d)\}$, as required.
\end{proof}

\begin{proof}[Proof of Lemma \ref{psl23fLem}]
We view $S=\PSL_2(3^f)$ as the subgroup of $\PGL_2(3^f)$ consisting of the images under the canonical projection $\GL_2(3^f)\rightarrow\PGL_2(3^f)$ of all matrices in $\GL_2(3^f)$ whose determinant is a square in $\IF_{3^f}$. Note that the order of every element of $S\alpha$ is divisible by $f$ and that by \cite[Proposition 4.1]{Har92a}, $(S\alpha)^f$ lies in some copy of $\PGL_2(3)\cong\Sym(4)$ inside the simple Chevalley group $A_1(\overline{\IF_3})$ containing $S$. In particular, the order of $(S\alpha)^f$ is an element of $\{1,2,3,4\}$, so we are done if we can show the following two statements:
\begin{itemize}
\item For all $s\in S$, $\ord((s\alpha)^f)\notin\{1,3\}$.
\item There is an $s\in S$ with $\ord((s\alpha)^f)=4$.
\end{itemize}
Let us start with the first statement. Write $\alpha=s'\delta\phi$ where $s'\in S$, $\delta$ is any fixed element of $\PGL_2(3^f)\setminus\PSL_2(3^f)$, and $\phi$ is a field automorphism of order $f$ (not necessarily the entry-wise Frobenius automorphism $x\mapsto x^3$). Then for each $s\in S$,
\[
(s\alpha)^f=(ss'\delta\phi)^f=(ss'\delta)(ss'\delta)^{\phi}\cdots(ss'\delta)^{\phi^{f-1}},
\]
and so, since $ss'\delta\in\PGL_2(3^f)\setminus\PSL_2(3^f)$ and $f$ is odd, it follows that the order of $(s\alpha)^f$ is even. This concludes the proof of the first statement.

For the second statement, denote again by $\phi$ the common field part of the elements of $S\alpha$. Since $\PGL_2(3)\cong\Sym(4)$, we have that $\PGL_2(3)\setminus\PSL_2(3)=\PGL_2(3)\setminus\PGL_2(3)'$ contains an element $g\zeta\GL_2(3)$ of order $4$. Observe that the lift $g\in\GL_2(3)$ of $g\zeta\GL_2(3)$ must have determinant $-1$, for its determinant must be a non-square in $\IF_3$. But since $f$ is odd, $-1$ is also a non-square in $\IF_{3^f}$, so $\beta:=g\zeta\GL_2(3^f)$ lies in $\PGL_2(3^f)\setminus\PSL_2(3^f)$ and also has order $4$. Since $\beta$ is centralized by $\phi$ and $\gcd(4,f)=1$, it follows that $\beta\phi\in S\alpha$ has order $4f$, as required.
\end{proof}

We are now ready for the

\begin{proof}[Proof of Theorem \ref{newMBTheo}]
Let us start with the proof of statement (2), because it is shorter and easier. Firstly, note that since $x^8$ is not MB by \cite[Theorem 1.1.2(1)]{Bor17a}, we also have that $x^{8k}$ is not MB for any $k\in\IN^+$ (if a finite group satisfies a probabilistic identity with respect to $x^8$ and $\rho$, it also satisfies one with respect to $x^{8k}$ and $\rho$). We may thus assume that $e\equiv4\Mod{8}$; in other words, $e=4f$ for some odd $f\in\IN^+$ with $f>1$. But by Lemma \ref{psl23fLem}, if $\alpha\in\Aut(\PSL_2(3^f))$ is as in the formulation of Lemma \ref{psl23fLem}, then $(\PSL_2(3^f)\alpha)^{e}=\{1\}$, whence $e$ is bad by \cite[Proposition 2.9(3)]{Bor17a}.

We now give the proof of statement (1). First, note that the assumption implies that each prime divisor of $m$ is larger than $226$. We need to show that for every nonabelian finite simple group $S$ and all $\alpha\in\Aut(S)$, $(S\alpha)^{2m}\not=\{\alpha^{2m}\}$. By \cite[Theorem 5.1]{Bor17a}, it suffices to show this for $S$ of one of the two forms $\PSL_2(p^f)$ or $\Suz(2^{2k+1})$. In what follows, we denote by $\Phi_S$ the field automorphism group of $S$, which is cyclic, generated by $\phi$, the entry-wise Frobenius automorphism $a\mapsto a^p$ (for this to make sense in the Suzuki case, view $\Suz(2^{2k+1})$ as a subgroup of $\GL_4(2^{2k+1})$ as in \cite{Suz60a}). As in the proof of Lemma \ref{psl23fLem} above, we view $\PSL_2(p^f)$ as a subgroup of $\PGL_2(p^f)=\Inndiag(\PSL_2(p^f))$, and we denote the image of a matrix $M\in\GL_2(p^f)$ under the canonical projection $\GL_2(p^f)\rightarrow\PGL_2(p^f)$ by $\overline{M}$.
\begin{enumerate}
\item Case: $S=\Suz(2^{2k+1})$. Then $\Out(S)=\Phi_S=\langle\phi\rangle$, so we may choose $\alpha=\phi^t$ for some $t\in\{0,\ldots,2k\}$. Then $\alpha$ centralizes $\Frob(20)\cong\Suz(2)\leq S$. In particular, there is an $s\in S$ of order $5\nmid 2m$ centralized by $\alpha$. It follows that $(s\alpha)^{2m}=s^{2m}\alpha^{2m}\not=\alpha^{2m}$.
\item Case: $S=\PSL_2(p^f)$. We make a subcase distinction:
\begin{enumerate}
\item Subcase: $p=2$. Then $\Out(S)=\Phi_S=\langle\phi\rangle$, so we may choose $\alpha=\phi^t$ for some $t\in\{0,\ldots,f-1\}$. Then $\alpha$ centralizes $\Sym(3)\cong\PSL_2(2)\leq S$. In particular, there is an $s\in S$ of order $3\nmid 2m$ centralized by $\alpha$. It follows that $(s\alpha)^{2m}=s^{2m}\alpha^{2m}\not=\alpha^{2m}$.
\item Subcase: $p>2$. Then
\[
\Out(S)=\Outdiag(S).\Phi_S=\langle\overline{\begin{pmatrix}\xi & 0 \\ 0 & 1\end{pmatrix}}S\rangle.\langle\phi\rangle
\]
where $\xi$ is some fixed generator of $\IF_{p^f}^{\ast}$. We may thus choose $\alpha=\overline{\begin{pmatrix}\xi & 0 \\ 0 & 1\end{pmatrix}}^{\epsilon}\phi^t$ for some $\epsilon\in\{0,1\}$ and some $t\in\{0,\ldots,f-1\}$. If $\epsilon=0$, then we can conclude as in Subcase 1, using that $\PSL_2(p)$ contains an element of order $3$. So assume that $\epsilon=1$. We make a subsubcase distinction:
\begin{enumerate}
\item Subsubcase: $p\geq7$ or $\gcd(f,t)>1$. Note that the centralizer of $\alpha$ in $\Inndiag(S)=\PGL_2(p^f)$ contains the element
\begin{align*}
\alpha^{\ord(\phi^t)}&=\alpha^{f/\gcd(f,t)}=\overline{\begin{pmatrix}\prod_{k=0}^{f/\gcd(f,t)-1}{\xi^{p^{kt}}} & 0 \\ 0 & 1\end{pmatrix}} \\
&=\overline{\begin{pmatrix}\prod_{k=0}^{f/\gcd(f,t)-1}{\xi^{p^{k\gcd(f,t)}}} & 0 \\ 0 & 1\end{pmatrix}}=\overline{\begin{pmatrix}\xi^{\frac{p^f-1}{p^{\gcd(f,t)}-1}} & 0 \\ 0 & 1\end{pmatrix}},
\end{align*}
whose order is $p^{\gcd(f,t)}-1$. In particular, since $[\Inndiag(S):S]=2$, there is an $s\in S$ of order $(p^{\gcd(f,t)}-1)/2$ centralized by $\alpha$. We will now argue that $(p^{\gcd(f,t)}-1)/2$ does not divide $2m$, then we can conclude as in Subcase 1. To that end, note that by the subsubcase assumption, $(p^{\gcd(f,t)}-1)/2>2$, so it suffices to show that $(p^{\gcd(f,t)}-1)/2$ is not of the form $n$ or $2n$ for some $n>1$ that is odd and satisfies the congruence $n\equiv 1\Mod{225}$.
\begin{itemize}
\item If $\frac{p^{\gcd(f,t)}-1}{2}=n$: Then $2n+1=p^{\gcd(f,t)}$. By assumption, $n\equiv1\Mod{3}$, so that $3\mid 2n+1$ and thus $p=3$. But $2n+1>3$, so one would need to have $9\mid 2n+1$, which is impossible since $n\equiv1\Mod{9}$ by assumption.
\item If $\frac{p^{\gcd(f,t)-1}}{2}=2n$: Then $4n+1=p^{\gcd(f,t)}$. By assumption, $n\equiv1\Mod{5}$, so that $5\mid 4n+1$ and thus $p=5$. But $4n+1>5$, so one would need to have $25\mid 4n+1$, which is impossible since $n\equiv1\Mod{25}$ by assumption.
\end{itemize}
\item Subsubcase: $p\in\{3,5\}$ and $\gcd(f,t)=1$ (i.e., $\phi^t$ is a generator of $\Phi_S$). By Lemma \ref{nonConstantLem}, it suffices to show that $(S\alpha)^{2m}\not=\{1\}$. Since $f=\ord(\phi^t)\mid\ord(s\alpha)$ for all $s\in S$, this is clear if $f\nmid 2m$, so assume that $f\mid 2m$. Note that by the argument from the previous subsubcase, we always have that $(p-1)f\mid\ord(\alpha)$. In particular, if $p=5$, or if $p=3$ and $f$ is even, then $4\mid\ord(\alpha)$, so that $\ord(\alpha)\nmid 2m$ and we are done. So we may assume that $p=3$ and $f$ is odd. But then Lemma \ref{psl23fLem} yields that some element in $S\alpha$ has order divisible by $4$; in particular, the $(2m)$-th power of that element is nontrivial, as required.
\end{enumerate}
\end{enumerate}
\end{enumerate}
\end{proof}

\end{document}